\documentclass[12pt]{amsart}
\usepackage{amssymb}
\usepackage[english]{babel}
\usepackage{enumerate}
\usepackage{dsfont}

\def\C{{\mathbb C}}
\def\T{{\mathbb T}}
\def\R{{\mathbb R}}

\def\D{{\mathbb D}}
\def\LL{{\mathcal L}}

\newcommand{\norm}[1]{\left\Vert#1\right\Vert}

\DeclareMathOperator{\BMOA}{BMOA}
\DeclareMathOperator{\VMOA}{VMOA}
\DeclareMathOperator{\LMOA}{LMOA}

\newtheorem{theorem}{Theorem}[section]
\newtheorem{proposition}[theorem]{Proposition}

\newtheorem{corollary}[theorem]{Corollary}

\newtheorem{assumption}{Assumption}

\theoremstyle{definition}

\newtheorem{example}[theorem]{Example}

\begin{document}

\title{Weak compactness of operators acting on \MakeLowercase{o}--O type spaces}

\author{Karl-Mikael Perfekt}

\address{%
  Department of Mathematical Sciences \\
  Norwegian University of Science and Technology \\
  7491 Trondheim
}


\begin{abstract}
We consider operators $T : M_0 \to Z$ and $T : M \to Z$, where $Z$ is a Banach space and $(M_0, M)$ is a pair of Banach spaces belonging to a general construction in which $M$ is defined by a "big-$O$" condition and $M_0$ is given by the corresponding "little-$o$" condition. Prototype examples of such spaces $M$ are given by $\ell^\infty$, weighted spaces of functions or their derivatives, bounded mean oscillation, Lipschitz-H\"older spaces, and many others. The main result characterizes the weakly compact operators $T$ in terms of a certain norm naturally attached to $M$, weaker than the $M$-norm, and shows that weakly compact operators $T : M_0 \to Z$ are already quite close to being completely continuous. Further, we develop a method to extract $c_0$-subsequences from sequences in $M_0$. Applications are given to the characterizations of the weakly compact composition and Volterra-type integral operators on weighted spaces of analytic functions, $BMOA$, $VMOA$, and the Bloch space. 
\end{abstract}

\maketitle 
\section{Introduction}
Let $Z$ be a Banach space. The main result of this paper characterizes the weak compactness of operators $T : M_0 \to Z$ and $T : M \to Z$, where $(M_0, M)$ is a pair of Banach spaces in which $M$ is defined by a "big-$O$" condition and $M_0$ by the corresponding "little-$o$" condition. See \eqref{eq:M} and \eqref{eq:M0} for the precise definition. The class of spaces $(M_0, M)$ is large and examples include $c_0$ and $\ell^\infty$, weighted and the corresponding vanishing weighted spaces of continuous, analytic or harmonic functions, M\"obius invariant spaces of analytic functions, Lipschitz-H\"older spaces, bounded and vanishing mean oscillation (BMO and VMO), and several others. The pair $(M_0, M)$ was first introduced in \cite{Perf13}, and the quoted examples are given there.  

This paper is inspired by recent works on the compactness properties of composition and integral operators acting on specific examples of spaces $M_0$ and $M$ \cite{Cont14}, \cite{Lait11}, \cite{Lait13}, \cite{Lefe08}. It often turns out that weak compactness and compactness are equivalent for these classes of operators, a phenomenon which can be readily understood given the main results of this article.

For the statements of the theorems, note that $M$ is associated with a reflexive Banach space $X$ in which $M$ is continuously contained (see Section \ref{sec:prelim}). For instance, $\ell^\infty$ is continuously contained in a weighted $\ell^2$-space.

{
\renewcommand{\thetheorem}{\ref{thm:main}}
\begin{theorem}
  A bounded operator $T : M_0 \to Z$ is weakly compact if and only if there for each $\varepsilon > 0$ exists an $N > 0$ such that
  \begin{equation} 
  \|Tx\|_Z \leq N \|x\|_X + \varepsilon \|x\|_{M}, \quad x \in M_0.
  \end{equation}
\end{theorem}
\addtocounter{theorem}{-1}
}
A similar description of the weakly compact operators on $C(K)$-spaces was given by Niculescu, and a far-reaching generalization to operators acting on general $C^*$-algebras is due to Jarchow \cite{Jarc86}. More recently, characterizations in the same spirit have been given for operators acting on $H^\infty$ (\cite{Lefe05}) and certain subspaces of Orlicz spaces  \cite{Lefe08}. 

In \cite{Perf13} it was proven that $M_0^{**} \simeq M$ in a canonical way. Therefore, Theorem \ref{thm:main} also applies to operators $T : M \to Z$ such that $(T|_{M_0})^{**} = T$ -- i.e. operators $T$ which are $\textrm{weak}^*$-$\textrm{weak}$ continuous, a continuity property which is simple to verify in many concrete examples. See Corollary \ref{cor:main}.

To compare the weak compactness characterization with compactness criteria, note that $T : M_0 \to Z$ is completely continuous if and only if for every bounded sequence $(x_n) \subset M_0$ such that $x_n$ converges weakly to zero, it holds that $\lim_n \|Tx_n\|_Z = 0$. To demand instead only weak compactness, one simply replaces the weak convergence of $x_n$ with the stronger  property (see \cite{Perf13}) that $x_n$ converges to zero in $X$-norm. The two conditions on the sequence $(x_n)$ are in many concrete examples closely related; herein lies the explanation of why weak compactness and compactness often are equivalent for operators on $M_0$ and $M$. See the examples in Section \ref{sec:examples}. 

The motivation for the proof of Theorem \ref{thm:main} comes from \cite{Perf14}, where it was shown that $M_0$ is an M-ideal in $M$. In particular, weakly compact operators on $M_0$ can be characterized in terms of $c_0$-subspaces of $M_0$. The proof hence relies on a procedure to create $c_0$-subspaces, a construction which we summarize as a separate theorem. 

{
\renewcommand{\thetheorem}{\ref{thm:c0}}
\begin{theorem}
Suppose that $x_n \in M_0$, $n = 1,2,3,\ldots$, is a sequence such that $\|x_n\|_M = 1$ and $\lim_{n\to\infty}\|x_n\|_X = 0$. Then $(x_n)_n$ has a subsequence which, as a basic sequence in $M_0$, is equivalent to the canonical basis of $c_0$.
\end{theorem}
\addtocounter{theorem}{-1}
}
This result is classical for $M_0 = c_0$, and has also been proven for the case when $M_0 = VMO$ \cite{Leib86}, the latter fact which has been used in \cite{Lait11} and \cite{Lait13} to characterize the weak compactness of Volterra-type integral operators and composition operators on the analytic $BMO$-space.

The paper is organized as follows. In Section \ref{sec:prelim} the definitions of $M_0$ and $M$ are given, as well as technical preliminaries; in Section \ref{sec:results} the main results are proven; Section \ref{sec:examples} gives applications of Theorem \ref{thm:main} and its corollary to composition and integral operators on weighted spaces of analytic functions, Bloch spaces, and analytic $BMO$-spaces.

\section{Definitions and preliminaries} \label{sec:prelim}
The spaces $M$ and $M_0$ are defined by
\begin{equation}\label{eq:M}
M(X,\LL) = \left\{ x \in X \, : \, \sup_{L \in \LL} \norm{Lx}_Y < \infty \right\}
\end{equation}
and
\begin{equation} \label{eq:M0}
M_0(X,\LL) = \left\{ x \in M(X, \LL) \, : \, \varlimsup_{\LL \ni L\to\infty} \norm{Lx}_Y = 0 \right\}.
\end{equation}
Here $X$ and $Y$ are Banach spaces, where $X$ is assumed to be separable and reflexive. $\LL$ is a collection of continuous linear operators $L\colon X \to Y$ that is made into a topological space $(\LL, \tau)$ by a $\sigma$-compact locally compact Hausdorff topology $\tau$. The topology should respect the strong operator topology in the sense that for every $x \in X$, the map $T_x\colon \LL \to Y$ given by $T_x L = Lx$ is continuous. The limit $L \to \infty$ in the definition of $M_0$ should be understood in the sense of one-point compactification of $(\LL, \tau)$ (i.e. $L$ should escape all compact sets).

We may assume that $M(X,\LL)$ is dense in $X$ \cite{Perf13}, and we suppose that
\begin{equation*}
\norm{x}_{M(X,\LL)} = \sup_{L \in \LL} \norm{Lx}_Y
\end{equation*}
defines a norm on $M(X,\LL)$ which is stronger than the $X$-norm. As in the concrete examples mentioned in the introduction, we want to consider the situation where the bidual $M_0^{**}$ can be canonically identified with $M$. For this to be true it is necessary to impose the following approximation property.

\begin{assumption} \label{as1} For every $x\in M(X, \LL)$ there is a bounded sequence $\{x_n\}_{n=1}^\infty$ in $M_0(X,\LL)$ such that $x_n$ converges weakly to $x$ in $X$. 
\end{assumption}

Henceforth we always assume that \ref{as1} holds. There is also the stronger hypothesis:

\begin{assumption} \label{as2} For every $x\in M(X,\LL)$ there is a bounded sequence $\{x_n\}_{n=1}^\infty$ in $M_0(X,\LL)$ such that $x_n$ converges weakly to $x$ in $X$ and $\sup_n \norm{x_n}_{M(X,\LL)} \leq \norm{x}_{M(X,\LL)}$. 
\end{assumption}

The next theorem, stating that indeed $M_0^{**} = M$ holds, was proven in \cite{Perf13}. For its statement, note that $M_0(X,\LL)$ can be viewed as a closed subspace of both $M$ and $M_0^{**}$. 
 
\begin{theorem}[\cite{Perf13}] \label{thm:old}
The dual space $X^*$ is continuously contained and dense in $M_0(X, \LL)^*$. Denoting by $$I \colon  M_0(X, \LL)^{**} \to X$$ the adjoint of the inclusion map $J \colon  X^* \to M_0(X,\LL)^*$, the operator $I$ is a continuous isomorphism of $M_0(X, \LL)^{**}$ onto $M(X,\LL)$ which acts as the identity on $M_0(X,\LL)$. Furthermore, $I$ is an isometry if Assumption \ref{as2} holds. 
\end{theorem}

In the isometric case the author proved in \cite{Perf14} that $M_0$ is an \textit{M-ideal} in $M$. In particular, $M_0$ has Pe\l czy\'nski's property (V), which as a consequence gives the following characterization of weakly compact operators on $M_0$ (see \cite{Harm93}).

\begin{proposition} [\cite{Perf14}]
Suppose that Assumption \ref{as2} holds. If $Z$ is a Banach space and $T : M_0(X,\LL) \to Z$ is a bounded operator, then $T$ is weakly compact if and only if there does not exist a subspace $F \subset M_0(X, \LL)$ isomorphic to $c_0$ such that $T|_F$ is an isomorphism. 
\end{proposition}
The proof of Theorem \ref{thm:main} is inspired by this proposition, but technically only relies on its forward direction which follows easily for any Banach space from the fact that $c_0$ has the Dunford-Pettis property. 

A sequence $(z_n)_{n=1}^\infty$ in a Banach space $Z$ is called basic if it is a (Schauder) basis for its span $[z_n] = \textrm{span} \{z_n\}$. Two basic sequences $(z_n)$ and $(w_n)$ in Banach spaces $Z$ and $W$, respectively, are said to be equivalent if there is an isomorphism between $[z_n]$ and $[w_n]$ which maps $z_n$ onto $w_n$, for all $n$. In this situation, if $W = c_0$ and $(w_n)$ is the unit-vector basis of $c_0$, we say that $(z_n)$ is equivalent to the canonical basis of $c_0$. For rudimentary information about bases, we refer to the classical paper of Bessaga and Pe\l czy\'nski \cite{Bess58}, the techniques of which will be utilized to prove the main results of this paper. 

\section{Results and Proofs} \label{sec:results}
In the proof of Theorem \ref{thm:c0} we make use of the embedding operator $V : M_0(X, \LL) \to C_0(\LL, Y)$ which isometrically embeds $M_0$ into the space of continuous $Y$-valued functions on $\LL$ vanishing at infinity. Explicitly
\begin{equation*}
Vx(L) = Lx, \quad x \in M_0, \, L \in \LL.
\end{equation*}
\begin{theorem} \label{thm:c0}
Suppose that $x_n \in M_0(X, \LL)$, $n = 1,2,3,\ldots$, is a sequence such that $\|x_n\|_M = 1$ and $\lim_{n\to\infty}\|x_n\|_X = 0$. Then $(x_n)_n$ has a subsequence which, as a basic sequence in $M_0(X, \LL)$, is equivalent to the canonical basis of $c_0$. 
\end{theorem}
\begin{proof}
 We will construct a subsequence $(z_n)_n$ of $(x_n)_n$ inductively. We will also construct two auxiliary sequences; a strictly increasing sequence of positive integers $(\beta_n)_n$, and a sequence $(f_n)_n$ in $B(\LL, Y)$, the space of bounded Baire measurable $Y$-valued functions equipped with the supremum norm. To begin, let $z_1 = x_1$, $\beta_1 = 1$ and $f_1 = V z_1$. For the construction, fix a strictly increasing sequence $\mathcal{K}_1 \subset \mathcal{K}_2 \subset \cdots$ of compact Baire subsets of $(\LL, \tau)$ such that $\LL = \bigcup_{n=1}^\infty \mathcal{K}_n$. We denote by $\mathcal{K}_n^c$ the complement of $\mathcal{K}_n$ in $\LL$. 

Suppose now that $z_1, \ldots, z_{n-1}$, $\beta_1, \ldots, \beta_{n-1}$, and $f_1, \ldots, f_{n-1}$ have been chosen.  Since each $z_j$ belongs to $M_0(X,\LL)$ we can pick $\beta_n > \beta_{n-1}$ such that
\begin{equation}
\|L z_j \|_Y \leq 1/2^j, \quad L \in \mathcal{K}_{\beta_n}^c, \, j = 1, \ldots, n-1.
\end{equation}
Since the operators $ L \in \mathcal{K}_{\beta_n}$ are uniformly bounded by the Banach-Steinhaus theorem, it follows from $\lim_k \|x_k\|_X = 0$ that we may choose $z_n$ to be an element from $(x_k)_k$ such that
\begin{equation} \label{set:an}
\{L \in \LL \, : \, \|Lz_n\|_Y > 1/2^n \} \subset \mathcal{K}_{\beta_n}^c.
\end{equation}
Denoting the set on the left hand side of \eqref{set:an} by $\mathcal{A}_n$, let 
\begin{equation*}
f_n = \mathds{1}_{\mathcal{A}_n} Vz_n,
\end{equation*}
where $\mathds{1}_{\mathcal{A}_n}$ is the characteristic function of $\mathcal{A}_n$.  

With the inductive process complete, we now claim that $(z_n)_{n=2}^\infty \subset M_0(X,\LL)$ has a further subsequence  equivalent to the canonical basis of $c_0$. To see this, let 
\begin{equation*}
\mathcal{B}_n = \mathcal{A}_n \setminus \cup_{j > n} \mathcal{A}_j, \quad n \geq 2.
\end{equation*}
 If $L \in \mathcal{B}_m$ for some $m \geq 2$, then $f_n(L) = 0$ for $n > m$, while by construction
\begin{equation*}
\| f_n(L) \|_Y \leq 1/2^n, \, \textrm{for } n < m.
\end{equation*}
Since $\|x_k\|_{M} = 1$ for all $k$, we of course have that $\|f_m(L)\|_Y \leq 1$. Hence, for $L \in \mathcal{B}_m$ we have
\begin{equation*}
\sum_{n=2}^\infty \|f_n(L)\|_Y \leq 1 + \sum_{n=2}^{m-1} \frac{1}{2^n}  < 3/2.
\end{equation*}
On the other hand, if $L \in \left( \cup_k \mathcal{B}_k \right)^c$, then $f_n(L) = 0$ for every $n \geq 2$, since $\cup_k \mathcal{B}_k = \cup_k \mathcal{A}_k$. For the latter equality, note that no $L \in \LL$ can belong to infinitely many sets $\mathcal{A}_k$, since $\mathcal{A}_k \subset \mathcal{K}_{\beta_k}^c$.

We have hence shown that
\begin{equation*}
\sum_{n=2}^\infty \|f_n(L)\|_Y  < 3/2, \quad \forall L \in \LL. 
\end{equation*}
It follows in particular that $(f_n)_{n=2}^\infty$ is a weakly unconditionally Cauchy \cite{Dies84} sequence in $B(\LL, Y)$. Note also that each $f_n$ was constructed as to have supremum norm $1$, $\|f_n\|_\infty = 1$. By the Bessaga-Pe\l czy\'nski selection principle (C. 1. and Lemma 3 of \cite{Bess58}) there is hence a basic subsequence $(f_{n_k})_k$ equivalent to the canonical basis of $c_0$. But then there is a positive integer $K$ such that $(V z_{n_k})_{k \geq K}$ is also basic and equivalent to the canonical basis of $c_0$, since 
\begin{equation*}
\sum_{k=1}^\infty \| V z_{n_k} - f_{n_k} \|_\infty \leq \sum_{k=1}^\infty \frac{1}{2^{n_k}} < 1.
\end{equation*}
This proves that $(z_{n_k})_{k \geq K}$ is a subsequence of the desired type.
\end{proof}

Based on Theorem \ref{thm:c0} we now prove Theorem \ref{thm:main}.
\begin{theorem} \label{thm:main}
Let $Z$ be a Banach space. A bounded operator $T : M_0(X,\LL) \to Z$ is weakly compact if and only if there for each $\varepsilon > 0$ exists an $N > 0$ such that
\begin{equation} \label{eq:weakcpctineq}
\|Tx\|_Z \leq N \|x\|_X + \varepsilon \|x\|_{M}, \quad x \in M_0(X,\LL).
\end{equation}
\end{theorem}
\begin{proof}
Since $X$ is reflexive, the inclusion $j : M_0(X) \to X$ is a weakly compact map. Based on this observation, it is a relatively well known fact that having \eqref{eq:weakcpctineq} implies the weak compactness of $T$ (see e.g. Proposition 10 in \cite{Lefe08}).

In the converse direction, suppose that \eqref{eq:weakcpctineq} does not hold. Equivalently, there is an $\varepsilon > 0$ and a sequence $(x_n)_n \subset M_0(X,\LL)$ with $\|x_n\|_{M} = 1$ such that
\begin{equation*} 
\|Tx_n\|_Z > n \|x_n\|_X + \varepsilon.
\end{equation*}
The boundedness of $T$ then automatically imposes $\lim_n \|x_n\|_X = 0$. Therefore Theorem \ref{thm:c0} applies, so that by passing to a subsequence we may assume that $(x_n) \subset M_0$ is equivalent to the canonical basis of $c_0$. In particular $(x_n)$ is weakly unconditionally Cauchy in $M_0$, and hence $(Tx_n)$ is weakly unconditionally Cauchy in $Z$. Since also $\|T x_n \|_Z \geq \varepsilon$ for all $n$, it has, by the Bessaga-Pe\l czy\'nski selection principle, a further subsequence $(Tx_{n_k})$ which too is equivalent to the canonical basis of $c_0$. But then both $(x_{n_k})$ and $(Tx_{n_k})$ are equivalent to the canonical basis of $c_0$, and $T$ must act as an isomorphism between the two $c_0$-subspaces $[x_{n_k}] \subset M_0$ and $[Tx_{n_k}] \subset Z$. Hence $T$ could not be weakly compact, or the Dunford-Pettis property of $c_0$ would be violated.
\end{proof}

As a corollary of Theorem \ref{thm:main} we obtain the corresponding result for operators $T : M(X,\LL) \to Z$ which are $weak^*$-$weak$ continuous. The $weak^*$-topology of $M(X,\LL)$ referred to is the one induced by the duality in Theorem \ref{thm:old}. Hence, letting $I$ denote the map of Theorem \ref{thm:old} and $T_0$ the restriction $T_0 = T|_{M_0}$, we have that $weak^*$-$weak$ continuity of $T$ means precisely that $T_0^{**}I^{-1} = T$, which by abuse of notation typically is written as $T_0^{**} = T$. 
\begin{corollary} \label{cor:main}
Let $Z$ be a Banach space and $T : M(X,\LL) \to Z$ be a bounded and $weak^*$-$weak$ continuous operator. Then $T$ is weakly compact if and only if there for each $\varepsilon > 0$ exists $N > 0$ such that
\begin{equation} \label{eq:weakcpctineq2}
\|Tx\|_Z \leq N \|x\|_X + \varepsilon \|x\|_{M}, \quad x \in M(X,\LL).
\end{equation}
\end{corollary}
\begin{proof}
Let $T_0 = T|_{M_0}$. The continuity hypothesis can equivalently be stated as $T_0^{**} = T$. Hence it follows from Gantmacher's theorem that $T$ is weakly compact if and only if \eqref{eq:weakcpctineq} holds. 

It remains to see that \eqref{eq:weakcpctineq} implies \eqref{eq:weakcpctineq2}. Suppose that $\varepsilon, N > 0$ are such that \eqref{eq:weakcpctineq} holds and let $x \in M(X,\LL)$. We renorm $M_0^{**}$ by equipping it with the equivalent norm 
\begin{equation*}
\| I^{-1} x \|_{\textrm{alt}} = N \|x\|_X + \varepsilon \|x\|_{M}, \quad I^{-1} x \in M_0^{**}.
\end{equation*}
  Invoking the weak-star-metrizability of the unit ball of $M_0^{**}$ ($M_0^*$ is separable by Theorem \ref{thm:old}), it follows that there exists a sequence of points $x_n \in M_0(X, \LL)$ converging weak-star to $x$ such that
\begin{equation*}
N \|x_n\|_X + \varepsilon \|x_n\|_{M} \leq N \|x\|_X + \varepsilon \|x\|_{M}, \quad \forall n.
\end{equation*}
By the continuity of $T$, $T x_n$ converges weakly to $Tx$, and therefore
\begin{equation*}
\|Tx\|_Z \leq \varliminf_n \|Tx_n\|_Z \leq \varliminf_n \left( N \|x_n\|_X + \varepsilon \|x_n\|_{M} \right) \leq N \|x\|_X + \varepsilon \|x\|_{M}.
\end{equation*}
\end{proof}
\section{Examples} \label{sec:examples}
Our first example will be of a general nature, to illustrate the idea that when compactness for a class of operators can be determined through a testing condition, then Corollary \ref{cor:main} may sometimes be used to show that weak compactness and compactness are equivalent for the class.
\begin{example} \label{ex:gen}
Suppose that $\{T_\alpha \}_\alpha$ is a family of bounded $\textrm{weak}^*$-$\textrm{weak}$ continuous operators $T_\alpha : M(X,\LL) \to Z$, $Z$ a Banach space, and that there is a "testing sequence" $(x_n) \subset M(X,\LL)$ such that:
\begin{itemize}
\item the sequence $(x_n)$ is bounded in $M(X,\LL)$,
\item $\lim_n \|x_n\|_X = 0$, and
\item for every $\alpha$, $\lim_n \|T_\alpha x_n\|_Z = 0$ implies that $T_\alpha$ is compact.
\end{itemize}
Suppose now that $T_\alpha$ is weakly compact. Then Corollary \ref{cor:main} immediately implies that $T_\alpha x_n$ must tend to zero in $Z$, so that $T_\alpha$ is actually compact. Hence, in the above situation, an operator $T_\alpha$ is compact if and only if it is weakly compact if and only if $\lim_n \|T_\alpha x_n\|_Z = 0$.
\end{example}

We now turn to several concrete examples of composition and integral operators acting on spaces of analytic functions. For an analytic function $\varphi: \D\to \D$, $C_\varphi$ denotes the composition operator
\begin{equation*}
C_\varphi f(z) = f(\varphi(z)), \quad z \in \D,
\end{equation*} 
where $f$ is a holomorphic function on $\D$, $f \in \textrm{Hol}(\D)$. We begin by considering composition operators $C_\varphi$ on weighted spaces.

\begin{example}
Let $v : \D \to \R_+$ be a strictly positive, radial, continuous weight on $\D$ such that $\lim_{|z| \to 1} v(z) = 0$, and consider the weighted spaces of holomorphic functions
\begin{equation*}
H_v^\infty = \{f \in \textrm{Hol}(\D) \, : \, \sup_{z\in \D} |f(z)|v(z) < \infty \}
\end{equation*}
and
\begin{equation*}
H_v^0 = \{f \in \textrm{Hol}(\D) \, : \, \varlimsup_{|z| \to 1^-} |f(z)|v(z) = 0\}.
\end{equation*}
They can be realized within our framework \cite{Perf13}, with the role of $X$ taken on by the analytic Bergman space on the disc with weight $v^2$, $$X = L^2_a(v^2 \, dA, \D) =  L^2(v^2 \, dA, \D) \cap \textrm{Hol}(\D).$$ Here $dA = dx \, dy$ denotes area measure. The desired approximation property Assumption \ref{as2} can be verified by considering dilations $f(rz)$ of a function $f \in H_v^\infty$, $r < 1$ (see \cite{Bone98}). 

Let $$\tilde{u}(z) = \sup_{\|f\|_{H_v^\infty} \leq 1} |f(z)|$$ and associate with $v$ the weight $\tilde{v} = 1/\tilde{u}$. 
Then $\tilde{v}$ is a weight of the same type as $v$ and $H_v^\infty = H^\infty_{\tilde{v}}$
 isometrically (\cite{Bone98}). $v$ is called \textit{essential} if $v$ is comparable to $\tilde{v}$. Given also a weight $w$ of the same type as $v$, Bonet et. al. characterized in \cite{Bone98} the compact composition operators $C_\varphi : H_v^\infty \to H_w^\infty$. We utilize Theorem \ref{thm:main} to add also weak compactness to their description. For simplicity we suppose that both $v$ and $w$ are essential.

\begin{proposition}
The following are equivalent:
\begin{description}
\item[i)] $C_\varphi : H_v^\infty \to H_w^\infty$ is compact,
\item[ii)] $C_\varphi : H_v^0 \to H_w^0$ is compact,
\item[iii)] $\lim_{r \to 1^-} \sup_{|\varphi(z)| > r} \frac{w(z)}{v(\varphi(z))} = 0$ or $\overline{\varphi(\D)} \subset \D$,
\item[iv)] $\lim_{|z| \to 1^-} \frac{w(z)}{v(\varphi(z))} = 0$,
\item[i')] $C_\varphi : H_v^\infty \to H_w^\infty$ is weakly compact,
\item[ii')] $C_\varphi : H_v^0 \to H_w^0$ is weakly compact.
\item[iii')] $C_\varphi(H_v^\infty) \subset H_w^0$
\end{description}
\end{proposition} 
\begin{proof}
The equivalences of \textbf{i)}-\textbf{iv)} are established in \cite{Bone98}, as is the (trivial) verification that $C_\varphi$ is $\textrm{weak}^*$-$\textrm{weak}$ continuous. The equivalences between \textbf{i')}-\textbf{iii')} follow from Gantmacher's theorem. We hence only need to show that \textbf{i')} implies \textbf{iv)}, which we do by following the proof of \textbf{i)} implies \textbf{iv)} and applying the criterion given by Corollary \ref{cor:main}.

If \textbf{iv)} does not hold, there is a sequence $(z_n)$ in $\D$ converging to a point $z_0 \in \partial \D$ such that $w(z_n) \geq c v(\varphi(z_n))$ for all $n$, for some $c > 0$. Since $v$ is essential, we can choose $f_n$ such that $\|f_n\|_{H^\infty_v} = 1$ and $|f_n(\varphi(z_n))| \sim 1/v(\varphi(z_n))$. It has to hold that $|\varphi(z_n)| \to 1$, or \textbf{i')} would be contradicted; we may select non-negative integers $\alpha_n \to \infty$ such that $|\varphi(z_n)|^{\alpha_n} \geq 1/2$ for all $n$. Consider the functions $g_n = z^{\alpha_n} f_n$. Since $|z|^{ \alpha_n}$ tends pointwise to zero in $\D$, and $|f_n|v_n$ is uniformly bounded, it follows by dominated convergence that $g_n$ converges to zero in $X = L^2_a(v^2)$. However,
\begin{equation*}
\|C_\varphi g_n\|_{H^\infty_w} \geq |g_n(\varphi(z_n)) w(z_n)| \geq c|\varphi(z_n)|^{\alpha_n} |f_n(\varphi(z_n))| v(\varphi(z_n)) \gtrsim \frac{c}{2},
\end{equation*}
contradicting \eqref{eq:weakcpctineq2}.
\end{proof}
\end{example}
For the next examples we introduce the spaces $\BMOA$ and $\VMOA$ of analytic functions of bounded and vanishing mean oscillation on the unit disc $\D$. To fit them into our framework, for $a \in \D$ and $\lambda \in \T$, let $\phi_{a,\lambda}$ be the disc automorphism
\begin{equation*}
\phi_{a, \lambda}(z) = \lambda \frac{a-z}{1-\bar{a}z}.
\end{equation*} 

Further, let $X = Y = H^2 / \C$, where $H^2$ is the usual Hardy space on the disc, and let $\LL$ consist of all composition operators $L_{\phi_{a,\lambda}} : H^2/\C \to H^2/\C$, 
\begin{equation*}
 L_{\phi_{a,\lambda}} f = f \circ {\phi_{a,\lambda}}  - f({\phi_{a,\lambda}} (0)).
\end{equation*}
We equip $\LL$  with the topology of $\D \times \T$. Then
\begin{equation}
M(H^2/\C, \LL) = \BMOA, \quad M_0(H^2/\C,\LL) = \VMOA,
\end{equation}
see \cite{Perf13}. We also have the Bloch spaces $B$ and $B_0$,
\begin{equation}
M(L^2_a/\C, \LL) = B, \quad M_0(L^2_a/\C,\LL) = B_0,
\end{equation}
where $L^2_a = L^2(\D) \cap \textrm{Hol}(\D)$ is the standard analytic Bergman space on the disc.

\begin{example}
Let $\varphi: \D\to \D$ be an analytic function. Several concrete realizations of Example \ref{ex:gen} can be given by considering composition operators $C_\varphi$ acting on spaces of analytic functions. In \cite{Tjani96} it is shown that $C_\varphi : Z \to B$, where $Z = B$ or $Z = \BMOA$, is compact if and only if $\lim_{|a| \to 1} \|C_\varphi \phi_{a,\lambda} \|_{B} = 0$, yielding that $C_\varphi : Z \to B$ is weakly compact if and only if compact.   If $\varphi \in B_0$, then $C_\varphi$ acts boundedly on $B_0$, and it follows in combination with Gantmacher's theorem that $C_\varphi : B_0 \to B_0$ is weakly compact if and only if compact, a result first shown in \cite{Madi95}. A more intricate example where Example \ref{ex:gen} applies is provided by \cite{Lait13}. Namely, $C_\varphi : \BMOA \to \BMOA$ is (weakly) compact if and only if $\lim_{|a| \to 1} \|C_\varphi \phi_{a,\lambda} \|_{\BMOA} = 0$.
\end{example}
The study of compact composition operators is well-developed. In recent contributions to the field, e.g. \cite{Cont14}, \cite{Lait13}, \cite{Lefe10}, the use of Banach space techniques has been essential. In fact, something reminiscent of Theorem \ref{thm:c0} often plays an important role. 

We conclude with an example of integral operators. The symbols of the operators will belong either to the logarithmic $\BMOA$-space $\LMOA = M(H^2/\C, \mathcal{K})$, or its corresponding small space $\LMOA_0 = M_0(H^2/\C, \mathcal{K})$. Here $\mathcal{K}$ consists of the weighted composition operators
\begin{equation*}
 K_{\phi_{a,\lambda}} f = \log \frac{2}{1-|a|} \left [ f \circ {\phi_{a,\lambda}}  - f({\phi_{a,\lambda}} (0))\right].
\end{equation*}

\begin{example}
For an analytic function $g$ in $\D$, we denote by $T_g$ the Volterra-type operator
\begin{equation}
T_g f (z) = \int_0^z f(\zeta) g'(\zeta) \, d \zeta, \quad z \in \D,
\end{equation}
acting on analytic functions $f$ in $\D$. Siskakis and Zhao \cite{Sisk99} showed that $T_g : \BMOA \to \BMOA$ is bounded if and only if $g \in \LMOA$. They proved in the same paper that $T_g : \BMOA \to \BMOA$ is compact if and only if $g \in \LMOA_0$, and posed the question whether $T_g : \BMOA \to \BMOA$ can be weakly compact without being compact. This was answered in the negative by Laitila, Mihkinen, and Nieminen \cite{Lait11}. The purpose of this example is to illustrate that the question may in fact be resolved using Siskakis and Zhao's original argument, when applied in conjuction with Corollary \ref{cor:main}.

First we point out that the boundedness of $T_g$, $g \in \LMOA$, automatically implies that $T_g(\VMOA) \subset \VMOA$, so that $T_g |_{\VMOA} : \VMOA \to \VMOA$ is a bounded operator. Secondly, it is easily verified that $\left( T_g |_{\VMOA} \right)^{**} = T_g$. That is, $T_g$ is $\textrm{weak}^*$-$\textrm{weak}$ continuous. By Gantmacher's theorem it follows that $T_g$ (or equivalently $T_g |_{\VMOA}$) is weakly compact if and only if $T_g(\BMOA) \subset \VMOA$.

When proving that compactness implies $g \in \LMOA_0$ in \cite{Sisk99}, the only step where compactness is used, as opposed to weak compactness, is in showing that $\lim_n \|T_g q_n \|_{\BMOA} = 0$, where 
\begin{equation*}
q_n(z) = \log \frac{1 - \bar{u}z}{1-\bar{u}_n z},
\end{equation*}
for a point $u \in \partial \D$ and a sequence $(u_n) \subset \D$ of points converging to $u$. However, $q_n$ is uniformly bounded in $\BMOA$, $$\|q_n \|_{\BMOA} \lesssim \|\log(1-z) \|_{\BMOA},$$ and $\lim_n \|q_n\|_{H^2} = 0$, so it follows from Corollary \ref{cor:main} that $T_g q_n \to 0$ in $\BMOA$, assuming only the weak compactness of $T_g$. With this remark in hand, one can follow the proof in \cite{Sisk99} verbatim to see that $T_g$ is weakly compact if and only if $g \in \LMOA_0$.
\end{example}

\end{document}